\documentclass[12 pt, reqno]{amsart}
\usepackage{fullpage}

\usepackage{amsmath}
\usepackage{amsfonts}
\usepackage{amssymb}
\usepackage{amsthm}
\usepackage{comment}
\usepackage{epsfig}
\usepackage{psfrag}
\usepackage{mathrsfs}
\usepackage{amscd}
\usepackage[all]{xy}
\usepackage{rotating}
\usepackage{lscape}
\usepackage{amsbsy}
\usepackage{verbatim}
\usepackage{moreverb}
\usepackage{color}
\usepackage{bbm}
\usepackage{eucal}
\usepackage{tikz}
\usepackage{tikz-cd}
\usetikzlibrary{decorations.markings,arrows}
\usetikzlibrary{cd}
\usepackage{hyperref}
\usepackage{xspace}
\usepackage{algorithm}
\usepackage[noend]{algorithmic}

\usepackage{chngcntr}
\counterwithin{figure}{section}
\newtheorem{thm}{Theorem}[subsection] 

\newtheorem{lemma}[thm]{Lemma}

\newtheorem{definition}[thm]{Definition}
\newtheorem{prop}[thm]{Proposition} 
\newtheorem{remark}[thm]{Remark}
\theoremstyle{plain}

\newtheorem*{thm*}{Theorem}

\newtheorem{example}[thm]{Example}

\newcommand{\lemref}[1]{Lemma~\ref{lem:#1}}

\newcommand{\ssref}[1]{subsection~\ref{ssec:#1}}
\newcommand{\figref}[1]{Figure~\ref{fig:#1}}
\newcommand{\exref}[1]{Example~\ref{ex:#1}}
\newcommand{\propref}[1]{Proposition~\ref{prop:#1}}
\newcommand{\defref}[1]{Definition~\ref{def:#1}}

\newcommand{\secref}[1]{Section~\ref{sec:#1}}

\theoremstyle{remark}

\definecolor{orange}{rgb}{.95,0.5,0}
\definecolor{light-gray}{gray}{0.75}
\definecolor{brown}{cmyk}{0, 0.8, 1, 0.6}
\definecolor{plum}{rgb}{.5,0,1}

\DeclareMathOperator{\Link}{\sf Link}

\usepackage{scalerel}
\usepackage{enumitem}

\DeclareMathOperator{\wwedge}{{\hstretch{.8}{\wedge\mkern-8mu\wedge}}}

\DeclareMathOperator{\colim}{{\sf colim}}

\DeclareMathOperator{\Top}{\mathsf{Top}}
\DeclareMathOperator{\Pos}{\mathsf{Poset}}

\DeclareMathOperator{\sCpx}{\mathsf{sCpx}}

\DeclareMathOperator{\Nat}{\sf Face}

\newcommand{\ehp}{the EHP hypotheses\xspace}

\def\cA{\mathcal A}\def\cC{\mathcal C}

\def\cP{\mathcal P}
\def\cQ{\mathcal Q}\def\cR{\mathcal R}

\def\NN{\mathbb N}
\def\RR{\mathbb R}
\def\WW{\mathbb W}
\def\ZZ{\mathbb Z}

\usepackage{graphicx}
\graphicspath{{./}{../figs/}}

\begin{document}

\title{Regularity via Links and Stein Factorization}
\author{Ryan Grady
\and Anna Schenfisch}
 \thanks{The authors declare they have no conflict of interest.}

\address{Department of Mathematical Sciences\\Montana State University\\Bozeman, MT 59717}
\email{ryan.grady1@montana.edu}

\address{Department of Mathematical Sciences\\Montana State University\\Bozeman, MT 59717}
\email{annaschenfisch@montana.edu}


\begin{abstract}
Here, we introduce a new definition of regular point for piecewise-linear (PL) functions on combinatorial (PL triangulated) manifolds.  This definition is given in terms of the restriction of the function to the link of the point.  We show that our definition of regularity is distinct from other definitions that exist in the combinatorial topology literature.  Next, we stratify the Jacobi set/critical locus of such a map as a poset stratified space.  As an application, we consider the Reeb space of a PL function, stratify the Reeb space as well as the target of the function, and show that the Stein factorization is a map of stratified spaces.
\end{abstract}


\maketitle

\setcounter{tocdepth}{2}
\tableofcontents
\thispagestyle{empty}

\section{Introduction}

In this work, we present a novel regularity condition for piecewise-linear (PL) functions on combinatorial (PL triangulated) manifolds. The idea is to recover the topological implications of the (Canonical) Submersion Theorem at a regular point, while not insisting on smoothness (or even differentiability).  

To wit, if $f \colon X^n \to \RR^k$ is a (generic) PL function from an $n$-dimensional combinatorial manifold (with $n \ge k$), we use the triangulation of $X$ to analyze the behavior of $f$ when restricted to the link of a point $p \in X$.  Inspired by the smooth setting, a point will be regular if $f$ respects a join presentation of the link into two spheres (of specified dimensions).  These spheres correspond to the unit spheres in the tangent and normal bundles to $X$ at $p$.  We use this {\em link regularity} to define the Jacobi set/critical locus: $J_f \subset X$, see Definition \ref{def:jacobi}.

In Section \ref{sec:reg}, we prove some elementary properties of the Jacobi set.  In particular, we equip $J_f$ with a stratification compatible with the ambient stratification of $X$. Moreover, we compare our regularity condition with two other notions in the literature: an approach via the PL differential and a homological regularity criterion.  We prove---and provide examples to illustrate---that these three definitions of regular point are distinct.

Throughout the paper, we work with poset stratified spaces.  To our knowledge, these spaces were first described in work of Lurie \cite{lurie}. Section \ref{sec:stratspaces} introduces the definition of poset stratified space and explicitly notes some elementary properties that are difficult to find in the literature. In \secref{otherstrat}, we (briefly) compare poset stratified spaces with the more classical notion of stratified space, e.g., as presented in \cite{Friedman}.

This project was motivated by the work of Edelsbrunner, Harer, and Patel \cite{edelsbrunner2008reeb} on Reeb spaces. Indeed, the initial goal was to express their ``canonical stratification" of the Reeb space in terms of poset stratified spaces.  To this end, we prove the following in Section \ref{sec:reeb}.

\begin{thm*}
Let $X$ be a combinatorial $n$-manifold and $f \colon X\to \RR^k$ a generic PL function, with $n \ge k$.  Let $\WW_f$ denote the corresponding Reeb space and $J_f$ the Jacobi set of $f$ (with regards to link regularity). Then,
\begin{itemize}
\item[{\rm (a)}] The stratification of $J_f$ induces a stratifications of $\WW_f$ and $\RR^k$ such that the Stein factorization $g \colon \WW_f \to \RR^k$ is a stratified map.
\item[{\rm (b)}] Further, if $f$ is injective on the $(k-1)$-simplicial skeleton of $X$, then we have a commutative diagram in stratified spaces
\begin{center}
\begin{tikzcd}[column sep=small]
X \arrow[rr, "f"] \arrow[dr, "q"'] & & Y \\
& \WW_f \arrow[ur, "g"'] & 
\end{tikzcd}
\end{center}
\end{itemize}
\end{thm*}

As we discuss in the final section of the paper, morally, the stratification of $\WW_f$ is a refinement of that in \cite{edelsbrunner2008reeb}.  However, as we discuss, there are some interesting subtleties.  Moreover, the stratification of $\RR^k$ and the verification that the Stein factorization is a stratified map is not contained in \emph{ibid}.

\subsection*{Acknowledgements} AS is supported by the National Science Foundation under NIH/NSF DMS 1664858. RG is supported by the Simons Foundation under Travel Support/Collaboration 9966728. The authors thank David Ayala for discussion and shared insight.  Many thanks also go to the anonymous referee for their thorough and incredibly helpful suggestions.

\section{Stratified Spaces}\label{sec:stratspaces}

Stratified spaces have been used throughout topology for over 60 years with major foundational contributions due to Thom, Mather, Goresky and MacPherson. The paradigm we use is developed extensively in the work of Ayala, Francis, and Tanaka \cite{aft} and first appeared in Appendix A of Lurie \cite{lurie}.

\subsection{Poset Stratified Spaces}

\subsubsection{Poset Topologies and Operations}

\begin{definition}
    A {\em poset} is a pair $(\cP, \le)$ consisting of
    \begin{itemize}
        \item a set $\cP$, and
        \item a relation $\le$ on $\cP$
    \end{itemize}
    such that $\le$ defines a partial order on $\cP$. I.e., for all $p,q,r \in \cP$, the following hold
    \begin{enumerate}
        \item (reflexivity) $p \le p$.
        \item (transitivity) if $p \le q$ and $q \le r$, then $p \le r$.
        \item (anti-symmetry) if $p \le q$ and $q \le p$, then $p=q$.
    \end{enumerate}
    A morphism of posets is an order-preserving map. The resulting category of posets is denoted $\Pos$.
\end{definition}

\begin{definition}
Let $(\cP, \le)$ be a poset. We equip $\cP$ with the {\em upward closed topology} as follows: $U \subseteq \cP$ is open if and only if for all $u \in U$, $\cP_{u \le} := \{ p \in \cP ~|~ u \le p\} \subseteq U.$
\end{definition}

\begin{prop}\label{prop:uc}
The upward closed topology defines a functor, $\mathrm{UC} : \Pos \to \Top$.
\end{prop}

\begin{remark}
Note that for $\cP \in \Pos$, the topology $\mathrm{UC} (\cP)$ is rarely Hausdorff.  Indeed, if two points, $x,y \in \cP$,  have a common upper bound, $x \le z \ge y$, they cannot be separated by disjoint opens.
\end{remark}

\begin{definition}
    Let $(\cP,\le)$ and $(\cQ,\le)$ be posets. We can explicitly describe the \emph{product poset}
    $(\cP \times \cQ,\le)$ by declaring
    $(p,q) \le (p',q')$ if $p\le p'$ and $q \le q'$.
\end{definition}

This construction is categorical.

\begin{prop}
The category $\mathsf{Poset}$ has (finite) products.
\end{prop}

\begin{example}
Let $(\RR, \le)$ be the standard total order on $\RR$. Then $\RR^2 \cong \RR \times \RR$ admits three partial orders
\begin{enumerate}
\item {\em Lexicographic order}: $(a,b) \le (c,d)$ if and only if $a<c$, or $a=c$ and $b \le d$;
\item {\em Product order}: $(a,b) \le (c,d)$ if and only if $a\le c$ and $b\le d$;
\item {\em Closure of direct product}: $(a,b) \le (c,d)$ if and only if $a <c$ and $b<d$, or $a=c$ and $b=d$.
\end{enumerate}
Only the first partial order is a total order. As the name indicates, the product order is the categorical product.
\end{example}

\begin{definition}
Let $(\cP, \le)$ be a poset. The {\em left cone} on $\cP$ is a new poset, $\cP^\triangleleft$, that is defined by adjoining a new minimum element to $\cP$. That is, let $[n]$ denote the totally ordered set $\{0 \le 1 \le \dotsb \le n\}$, then $\cP^\triangleleft$ is defined as the pushout

\[
\begin{tikzcd}
    \{0\} \times \cP \arrow[hookrightarrow]{r} \arrow[d,twoheadrightarrow,"\text{pr}_1"] & {[1]} \times \cP \arrow[d] \\
    \{0\} \arrow[r] & \cP^\triangleleft \arrow[ul, phantom, "\lrcorner", very near start].
    \end{tikzcd}
\]

Similarly, we can define the {\em right cone} on $\cP$ as the pushout
\[
\begin{tikzcd}
    \{1\} \times \cP \arrow[hookrightarrow]{r} \arrow[d,twoheadrightarrow,"\text{pr}_1"] & {[1]} \times \cP \arrow[d] \\
    \{1\} \arrow[r] & \cP^\triangleright \arrow[ul, phantom, "\lrcorner", very near start],
    \end{tikzcd}
\]
which adjoins a new maximum to $\cP$.
\end{definition}

\begin{example}
Note that $\emptyset^\triangleleft \cong \{0\}$,$\{0\}^\triangleleft \cong [1]$, and more generally, for $n \in \NN$,
$[n]^\triangleleft \cong [n+1] \cong [n]^\triangleright$.
\end{example}

\subsubsection{The category of stratified spaces}

\begin{definition}
A {\em stratified topological space} is a triple $(X \xrightarrow{\phi} \cP)$ consisting of
\begin{itemize}
    \item a paracompact, Hausdorff topological space, $X$,
    \item a poset $\cP$, equipped with the upward closed topology, and
    \item a continuous map $X \xrightarrow{\phi} \cP$.
\end{itemize}
\end{definition}

\begin{definition} Given a stratified topological space $\phi : X \to \cP$, and any $p \in \cP$, the {\em $p$-stratum}, 
$X_p$, is defined as
$$X_p := \phi^{-1}(p).$$
\end{definition}

\begin{example}
Any paracompact, Hausdorff space defines a stratified topological space via $\phi : X \to [0]$.
We call this the
\emph{trivial stratification}.
\end{example}

\begin{example}
$\RR$ is stratified over $\{- > 0 < +\}$, where $X_- = \{x \in \RR ~|~ x < 0\}$, $X_0 = \{0\}$, and
$X_+ = \{x\in \RR ~|~ x>0\}$ are mapped to $-$, $0$, and $+$, respectively.
We often find it convenient to indicate a stratified space pictorially, by drawing the strata. For example, below
is a depiction of the stratified space $(\RR \to \{- > 0 < +\})$:
    $$\begin{tikzpicture}[auto]
        \node (a) at (0:0) {$\bullet$};
        \node (b) at (0:1.5) {};
        \node (c) at (180:1.5) {};
        \node (d) at (270:.5) {$X_0$};
        \draw (a) to node {$X_+$} (b)
            (a) to  node[swap] {$X_-$} (c);
    \end{tikzpicture}$$
\end{example}

\begin{definition} 
A {\em map of stratified topological spaces} $(\phi : X \to \cP)$ to $(\psi : Y \to \cQ)$ is a pair of 
continuous maps $(f_1, f_2)$ making the following diagram commute.
$$\begin{tikzcd}
    X \arrow[r,"f_1"] \arrow[d,"\phi"] & Y \arrow[d,"\psi"] \\
    \cP \arrow[r,"f_2"] & \cQ.
    \end{tikzcd}$$
A stratified map is an {\em open embedding} if both $f_1$ and $f_2$ are open (topological) embeddings.
\end{definition}

We will let $\mathsf{TStrat}$ denote the category of stratified topological spaces and stratified maps. The category has many nice properties, e.g., it has finite products.

\begin{definition}
A {\em refinement} of a stratified space $(\phi : X \to \cP)$ is a stratified space $(\phi' : X \to \cR)$ and a map of stratified spaces $(f_1, f_2) : (X \to \cR) \to (X \to \cP)$ such that $f_1 : X \to X$ is the identity, and~$f_2 : \cR \to \cP$ is a continuous surjection, diagrammatically we have the following.
$$\begin{tikzcd}
    X \arrow[equal,r] \arrow[d,"\phi'"] & X \arrow[d,"\phi"] \\
    \cR \arrow[->>,r,"f_2"] & \cP.
    \end{tikzcd}$$
\end{definition}

\subsection{Filtered spaces and poset stratified spaces}\label{sec:otherstrat}

The definition of stratified space we are using is not the only such notion. We will now briefly recall the more classical definition of stratified space, which grew out of Thom's work on singularity theory, and note how the two notions compare.  We will follow the conventions of Greg Friedman's recent text \cite{Friedman}.

More classically, one starts with a Hausdorff space, $X$, filtered by closed subspaces
\[
\emptyset = X^{-1} \subseteq X^0 \subseteq X^1 \subseteq \dotsb \subseteq X^n = X.
\]
The {\em strata} of $X$ are the connected components of $X_i:= X^i \setminus X^{i-1}$ for $i \ge 0$. A canonical example is a finite simplicial complex filtered by its simplicial skeleton, i.e., $X^k$ consists of all simplices of dimension less than or equal to $k$. In this case, the $i$-strata consist of the interiors of all $i$-simplices in the complex.

In the setting of \cite{Friedman}, a filtered space is {\em stratified} if it satisfies the following {\em Frontier Condition}.

\begin{definition}
A filtered space $X$ satisfies the {\em Frontier Condition} if for any two strata $S$ and $T$, if $S \cap \overline{T} \neq \emptyset$, then $S \subset \overline{T}$.
\end{definition}

Let $\phi \colon X \to \cP$ be a stratified space in the sense of Lurie (and everywhere in this paper outside of this subsection). Further, assume that $\cP$ is a graded poset, i.e., is equipped with a map of posets (a rank function) $\rho \colon \cP \to \NN_0$. We can filter $X$ by sublevel sets of the composition $\rho \circ \phi \colon X \to \NN_0$.  As the following example shows, the resulting filtered space need not satisfy the Frontier Condition (potentially depending on the choice of rank function).  For more discussion and comparison between approaches to stratified spaces see \cite{aft, Nocera, Trotman}.

\begin{example}
Consider the {\em upright} height function on the torus $h: T^2 \to \RR$. 
There are four critical points and we stratify $T^2$ via the stable manifolds with respect to gradient flow of the function $-h$ (meaning we flow down). That is, for a critical point $c \in T^2$, the corresponding stratum, $X_c$, is given by
\[
X_c = \left \{ p \in T^2 \mid \lim_{t \to \infty} \Phi (t,p) = c \right \},
\]
where $\Phi$ is the gradient flow of $-h$.

More generally, returning the cart to behind the horse, let $M$ be a compact smooth manifold, and let $f: M \to \RR$ be a Morse function. The {\em stable$^{-}$ stratification} of $M$ is determined by the stratifying map $s^{-} : M \to \cC^{-}$, where $\cC^{-}$ is the poset of critical points $\{c_\alpha \}$ with generating relations: $c_\alpha$ and $c_\beta$ are related if there is a flow line from $c_\alpha$ to $c_\beta$.
Rephrasing  standard results in Morse theory, see \cite{liviu}, it follows that
 the triple $(M \xrightarrow{s^-} \cC^-)$ defines a stratified space in the sense of Lurie, i.e., that the stable$^-$ stratification is well defined. Moreover, this stable$^-$ stratification is Euclidean, meaning each stratum is diffeomorphic to $\RR^n$ for some $n$.

For our specific example of the upright torus, the poset of critical values is linear, and given by $\cC^{-}=\{t \le s_1 \le s_2 \le b\}$, where $t$, $s_1$, $s_2$, and $b$ are as in \figref{upright}. 
There is an obvious rank function, namely $\rho_1 \colon \cC^{-} \to \{0 \le 1 \le 2 \le 3\}$, where $\rho_1 (t)=0,  \rho_1 (s_1) = 1, \rho_1 (s_2) =2, \rho_1 (b)=3$.  Although this is a stratified space in the sense of Lurie, it is not a stratified space in the sense of Friedman, as the resulting filtered space does not satisfy the Frontier Condition. In particular, $X_1 \cap \overline{X_2} = \{s_1\}$ while $X_1 \not \subset \overline{X_2}$.

There is another natural rank function $\rho_2 \colon \cC^{-} \to \{0 \le 1 \le 2\}$, which sends a critical point to 2 minus its index.  With respect to this choice, the resulting filtered space \emph{does} satisfy the Frontier Condition.

\end{example}

\begin{figure}[h!] 
\centering
\includegraphics[scale=.65]{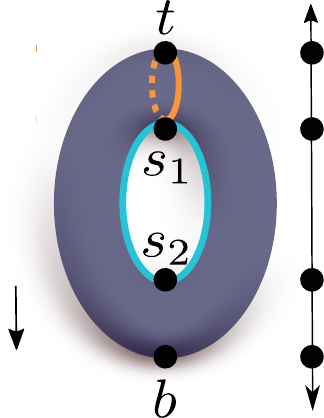}
\begin{caption}{The four critical values of $-h$ are shown on the upright $T^2$ along with their corresponding images in $\RR$. The direction of $-h$ is indicated by the leftmost arrow. Two strata defined by stable manifolds, $X_{s_1}$  and $X_{s_2}$, are shown in orange (upper circle) and teal (inner circle), respectively.}
\end{caption}
\label{fig:upright}
\end{figure}

\subsection{Face stratifications}

Simplicial complexes are canonically stratified in the sense of \cite{Friedman}.  We now show that they are very naturally poset-stratified as well.
The following definition gives a canonical stratification of a simplicial complex, the stratification induced by its \emph{face poset} (see, e.g.,~\cite{bjorner1984posets, stanley1991f, hersh2021shellability}).

\begin{definition}
Let $K$ be a simplicial complex. The {\em face stratification of $K$} is defined by $K \xrightarrow{\varphi_K} \Nat(K)$, where $\Nat(K)$ is the poset such that, for simplices $\tau, \sigma \subset K$, we have the relation $\tau \leq \sigma$ if and only if $\tau$ is a face of $\sigma$. Suppose that $\tau$ is the lowest dimensional simplex containing $x \in K$. Then $\varphi_K(x) = \tau$.
\end{definition}


We note a couple key properties of the face stratification.

\begin{prop}\label{prop:coarsest}
Let $K$ be a simplicial complex. Then $K \xrightarrow{\varphi_K} \Nat(K)$ is a refinement of the coarsest stratification of $K$ such that 
each stratum is a connected manifold.
\end{prop}

As discussed in the preceding subsection, when the stratifying poset is graded, we obtain a filtered space.

\begin{prop}\label{prop:filtration}
Let $K$ be a simplicial complex. There is a map of posets $\dim \colon \Nat(K) \to \NN_0$, associating to a simplex its dimension, such that the composite $\dim \circ \varphi_K \colon K \to \NN_0$ is the simplicial skeletal filtration.
\end{prop}

Let $\sCpx$ denote the category of simplicial complexes and simplicial maps. The following is immediate from Proposition \ref{prop:uc} and the definition of simplicial map.

\begin{prop}
Let $f \colon K \to L$ be a simplicial map, then there is a map of posets $\Nat(f) \colon \Nat(K) \to \Nat (L)$ that extends to a functor $\Nat\colon \sCpx \to \mathsf{TStrat}$.
\end{prop}

Note that the face stratification makes sense in other categories: $\Delta$-complexes, CW-complexes, simplicial objects, etc. 


\subsection{Ambient stratifications} 

We now consider how an embedded stratified space determines a stratification of the ambient space. The example of a simplicial complex---equipped with its face stratification---embedded in some Euclidean space is an illustrative example to track throughout.

Suppose that $(S \xrightarrow{ \varphi } \cQ)$ is a stratified space and $S \hookrightarrow X$ is a topological embedding into the paracompact, Hausdorff space $X$, e.g., $X = \RR^k$. (We will conflate $S$ with its image in $X$.) To extend the  stratification of~$S$ to a stratification of $X$ compatible with the usual topology on $X$, we can append a final element to the poset, $\cQ$, and map the complement,  $X \setminus Y$, to this element. This is exactly the result of taking the right cone on $\cQ$. The resulting stratification is the \emph{right conical extension} of $\varphi$, and we denote this map $\varphi_S \colon X \to \cQ^\triangleright$.

Note that we often want the strata to be connected. In this case we introduce a new stratifying poset $\cQ^{\wwedge}$, which is a refinement of $\cQ^\triangleright$.


\begin{definition}\label{def:connectedstrat}
Let  $(S \xrightarrow{ \varphi } \cQ)$ be a stratified space, $S \hookrightarrow X$ a topological embedding, and $\pi_0 (X \setminus S)= \cA$. Define the poset, $\cQ^{\wwedge}$, as the set $\cQ \amalg \cA$, subject to the following generating relations:
\begin{enumerate}
\item The relations of $\cQ$;
\item For $\ell \in \cQ$ and $\alpha \in \cA$, $\ell \le \alpha$ if and only if $\varphi^{-1}_S (\ell) \subseteq \overline{\alpha}$, i.e., the $\ell$-stratum is in the closure of the connected component indexed by $\alpha$.
\end{enumerate}
There is an obvious extension of the map $\varphi_S$, $\psi_S \colon X \to \cQ^{\wwedge}$ and we call this stratification the {\em connected ambient stratification}. 
\end{definition}

\begin{example}
Consider the 1-simplex embedded as $K\colon [0,1] \hookrightarrow \RR$. The poset $\Nat(K)$ consists of three objects: $b_0 \le a \ge b_1$ and $\Nat(K)^{\wwedge}$ is the poset $a_- \ge b_0 \le a \ge b_1 \le a_+$. The stratifying map $\psi_K \colon \RR \to \Nat(K)^{\wwedge}$ is given by
\[
\psi_K (x) = \begin{cases}a_-  & \text{if } x<0\\ b_0 & \text{if } x=0 \\ a & \text{if } x \in (0,1)\\ b_1& \text{if } x=1\\ a_+ & \text{if } x>1\end{cases} 
\]

\end{example}

\begin{lemma}\label{lem:faceext}
Let  $(S \xrightarrow{ \varphi } \cQ)$ be a stratified space, $S \hookrightarrow X$ a topological embedding, The maps $X \xrightarrow{\varphi_S} \cQ^\triangleright$ and $X \xrightarrow{\psi_S} \cQ^{\wwedge}$ define stratified spaces that restrict to the $\cQ$ stratification on $S$. Moreover, by collapsing the maximal elements in the poset, the connected ambient stratification refines the right conical extension of the $\cQ$ stratification:
\[
\xymatrix{ S\ar[d]^{\varphi} \ar@{^{(}->}[r] & X \ar@{=}[r] \ar[d]^{\psi_S} & X \ar[d]^{\varphi_S} \\ \cQ \ar[r]& \cQ^{\wwedge} \ar@{->>}[r]& \cQ^\triangleright}
\]
\end{lemma}

\begin{definition} Define a category $\mathsf{TStrat}^{\text{embd}}$ as follows:
\begin{itemize}
\item Objects: an object is a triple $(S, X, \iota_S \colon S \hookrightarrow X)$, where $(S \xrightarrow{\phi} \cP)$ is a stratified space, $X$ is a paracompact Hausdorff space, and~$\iota_S$ is a (topological) embedding. 
\item Morphisms: a morphism $f\colon (S,X,\iota_S) \to (T,Y, \iota_T)$ is a  map $f\colon X \to Y$ such that
\begin{enumerate}
\item The image of $S$ under $f$ is contained in $T$; 
\item The restriction $f|_S \colon (S \xrightarrow{\phi} \cP)  \to (T \xrightarrow{\phi'} \cQ)$ is a stratified map; and
\item The preimage of $T$ is contained in $S$, i.e., $f(X\setminus S) \cap T = \emptyset$.
\end{enumerate}
\end{itemize}
\end{definition}

We will at times condense notation and simply denote an object by $S \subset X$. Note that $\mathsf{TStrat}^{\text{embd}}$ is well defined as the identity is clearly a morphism and composition is inherited from composition in spaces. The conditions on $f\colon (S \subset X) \to (T \subset Y)$ are equivalent to asking that the following is a pull-back diagram
\[
\xymatrix{ S \ar[d]_{f|_S} \ar@{^{(}->}[r]^{\iota_S} & X \ar[d]^{f} \\ T \ar@{^{(}->}[r]^{\iota_T} &Y}
\]
This rephrasing makes composition apparent as pasting of pullback squares.

\begin{prop}\label{prop:facemap}
The connected ambient stratification defines a functor $(-)^{\wwedge} \colon \mathsf{TStrat}^{\text{embd}} \to \mathsf{TStrat}$. In particular, given a morphism $f\colon (S \subset X) \to (T \subset Y)$, we have a commutative diagram (in spaces) 
\[
\xymatrix{ X \ar[rr]^{f} \ar[d] && Y \ar[d] \\ \cP^{\wwedge} \ar[rr] && \cQ^{\wwedge}}
\]
\end{prop}
\begin{proof}
Lemma \ref{lem:faceext} shows that we have a well-defined functor at the level of objects.  As the map $f$ is continuous, it takes path components into path components, so we obtain a map at the level of morphisms. Functoriality follows, in a straightforward manner, from composition in $\mathsf{TStrat}$.
\end{proof}

\subsection{Combinatorial Manifolds}

Let us briefly recall some combinatorial and piecewise-linear (PL) topology. A standard reference is \cite{RSPL} or Section 3.9 of \cite{thurston}. (In the latter reference a combinatorial manifold is called a triangulated PL manifold.)

\begin{definition}
A \emph{combinatorial manifold}, $X$, is a PL manifold equipped with a compatible triangulation, i.e., a PL homeomorphism $K  \to X$ for some simplicial complex $K$.
\end{definition}

A PL manifold is a topological manifold equipped with an equivalence class of triangulations, while the combinatorial structure picks out a particular triangulation. It is no surprise that any PL manifold admits a structure of a combinatorial manifold, though the condition is still stronger than being a triangulated topological manifold. A classical theorem of Whitehead is that any smooth manifold admits a unique PL structure, so a triangulated smooth manifold has a unique structure as a combinatorial manifold.

\begin{definition} 
Let $X$ be a combinatorial manifold, and $i>1$. An $i$-tuple of points, $(p_1, \dotsc, p_i)$, in $X$ is in {\em general position} if the points are not contained in a common $(i-2)$-simplex of $X$ for any subdivision of the given triangulation. 
\end{definition}

\begin{definition}\label{defn:generic}
Let $X$ be a combinatorial manifold. A PL  (locally affine) function $f\colon X \to \RR^k$ is is \emph{generic} if for each $1<i\le k$ and any $i$-tuple of vertices $(v_1 , \dotsc , v_i)$ of $X$ in general position, the images $(f(v_1) , \dotsc , f(v_i))$ do not lie on the same $(i-2)$-plane in $\RR^k$. In other words, the images of vertices in general position are also in general position.
\end{definition}

\begin{remark}
Observe that genericity of a PL function $f \colon X \to \RR^k$ implies that $f$ restricted to a non-degenerate $n$-simplex of $X$ is injective. Consequently, the image of a non-degenerate $n$-simplex is an $n$-simplex in $\RR^k$.
\end{remark}

The {\em standard} PL disk is given by the cube $I^n$ and its boundary is the {\em standard} PL sphere. In actuality, any convex polyhedron of dimension $n$ presents the standard PL disk as they are all PL homeomorphic. A similar statement can be made for boundaries of convex polyhedrons and PL spheres.

\begin{definition}
Let $X$ and $Y$ be topological spaces. The \emph{join} of $X$ and $Y$, denoted $X \ast Y$, is given by the following colimit (iterated pushout)

\[
\colim \left(
\begin{tikzcd}
& X \times Y \ar[dl, "\text{pr}_X", swap] \ar[dr, "i_0"] && X \times Y \ar[dl, "i_1", swap] \ar[dr, "\text{pr}_Y"] & \\  X && X \times I \times Y && Y
\end{tikzcd}
\right),
\]

where $i_0$ is the inclusion $(x,y) \mapsto (x,0,y)$ and $i_1$ the inclusion $(x,y) \mapsto (x,1,y)$.
\end{definition}

Morally, the join of spaces is obtained by adjoining all line segments connecting the two spaces (this is literally true for polyhedral spaces embedded in Euclidean space). The join of simplicial complexes is again a simplicial complex, the join of PL manifolds is again PL, and the join of combinatorial manifolds is again a combinatorial manifold.
\begin{definition}
Let $K$ be a simplicial complex.
\begin{itemize}
\item For $\sigma \subseteq K$ a simplex, the \emph{star} of $\sigma$, $\mathop{St} \sigma$, is the set of simplices in $K$ that contain $\sigma$ as a face.
\item For $\sigma \subseteq K$ a simplex, the \emph{link} of $\sigma$, $\Link (\sigma)$, is the closure of the star of $\sigma$ set minus the union of stars of the faces of $\sigma$, i.e., $\Link(\sigma) = \mathop{Cl} \mathop{St} \sigma \setminus \mathop{St} \mathop{Cl} \sigma$.
\item For  $\sigma \subseteq K$ a simplex and $p \in \mathring{\sigma}$, the \emph{link} of $p$, $\Link (p)$, is given by
\[
\Link (p) = \partial \sigma \ast \Link (\sigma).
\]
\end{itemize}
\end{definition}

\begin{lemma}[\cite{thurston}]
A simplicial complex is a PL manifold if and only if the link of each simplex is PL homeomorphic to a PL sphere.
\end{lemma}

\begin{lemma}[\cite{RSPL}]
For $n \in \NN$, let $B^n$ be a PL $n$-disk,  $S^n$ a PL sphere, and $M$ any PL manifold. Assuming the factors are disjoint, we have the following PL homeomorphisms:
\[pt \ast M \cong CM, \quad
B^\ell \ast B^k \cong B^{\ell +k+1}, \; \; \text{ and } \; \; 
S^\ell \ast S^k \cong S^{\ell+k+1}.
\]
\end{lemma}

\section{Regularity via links}\label{sec:reg}

We now introduce a regularity condition on PL maps $f \colon X \to \RR^k$ by analyzing the restriction of $f$ to links of points in $X$.  This notion of regular/critical is compared to two other notions commonly used in combinatorial topology.

Throughout this section we assume $X$ is an $n$-dimensional combinatorial manifold and $f \colon X \to \RR^k$ is a generic PL function with $n \ge k$.

\subsection{Link Regularity}

Let $n \ge k$ and consider a smooth map $f \colon \RR^n \to \RR^k$. In the smooth setting, a point $p \in \RR^n$ is regular if the differential of $f$ at $p$, $df_p \colon T_p \RR^n \to T_{f(p)} \RR^k$, is onto. The  Submersion Theorem asserts there are local coordinates (at $p$) such that $f$ is simply projection onto the last $k$-factors. Let $B^n$ be a small PL ball centered at $p$ in this coordinate neighborhood, then we have a PL homeomorphism $B^n \cong I^{n-k} \times I^{k}$.  In this combinatorial neighborhood, we have that $\Link (p) \cong S^{n-1} \cong S^{n-k-1} \ast S^{k-1}$, where we identify the first factor with the unit normal sphere at $p$.
(Note that if $n=k$, then the normal bundle is empty and we adopt the convention that $S^{n-k-1} = \emptyset$.). The Submersion Theorem implies that restriction of $f$ to $\Link (p)$ is a homeomorphism when restricted to the unit tangent sphere, $S^{k-1}$, and collapses the normal sphere, $S^{n-k-1}$, to the point $f(p)$.

The goal is to keep the (combinatorial) topological implication of the Submersion Theorem, while removing the assumption of differentiability.\footnote{We thank the referee for some specific suggested language around this motivation.} Therefore, we make the following definition. Note that there is a small subtlety in the case that $n=k$: as $S^{n-k-1}=\emptyset$, there is no normal sphere to collapse onto the image point $f(p)$.

\begin{definition}\label{def:jacobi} Let $f\colon X^n \to \RR^k$ be a PL function with $n \ge k$.
\begin{itemize}
\item[{\rm (a)}] Suppose $n > k$.  A point $p \in X$ is \emph{regular} if there exists a PL homeomorphism $\Link (p) \cong S^{n-k-1} \ast S^{k-1}$ such that
\[
f |_{\Link (p)} \colon S^{n-k-1} \ast S^{k-1} \to f(p) \ast S^{k-1}\]
respects the join decomposition and is a PL homeomorphism of spheres on the second factor.
\item[{\rm (b)}] Suppose $n=k$.  A point $p \in X$ is \emph{regular} if there exists a PL homeomorphism $\Link (p) \cong  S^{k-1}$ such that $f |_{\Link (p)}$ is a homemorphism onto its image.
\item[{\rm (c)}] The \emph{Jacobi set} of $X$ is the collection of critical points, i.e., those points that are not regular. The Jacobi set of the map $f$ is denoted $J_f$.
\end{itemize}
\end{definition}

\begin{example}\label{ex:icos}
Let $X$ be an octohedron in $\RR^3$, and let $f\colon X \to \RR$ be the standard height function for a direction ${u} \in S^{2}$ such that all vertices of $X$ have a unique height with respect to $u$. Consider a point $p$, as in \figref{icos} We see that 
\[
\Link(p) = \Link ([am]) \ast \partial [am] = \{b, d\} \ast \{a,m\}
\] is the join of two copies of $S^0$, and the image $f|_{\Link(p)} = f(p) \ast \{f(a), f(m)\}$ is indeed a PL homeomorphism onto the second $S^0$ factor. Thus, we see that $p$ is regular.  Instead, consider $\Link(m) = \{a,c\} \ast \{d,b\}$. Although this link is also the join of two copies of $S^0$, since $f(m)$ is maximal, the image of $f|_{\Link(m)} \neq f(m) \ast S^0$. Thus, we see that $m$ is critical.
\end{example}

\begin{figure}[h!]
\centering
\includegraphics[scale=.4]{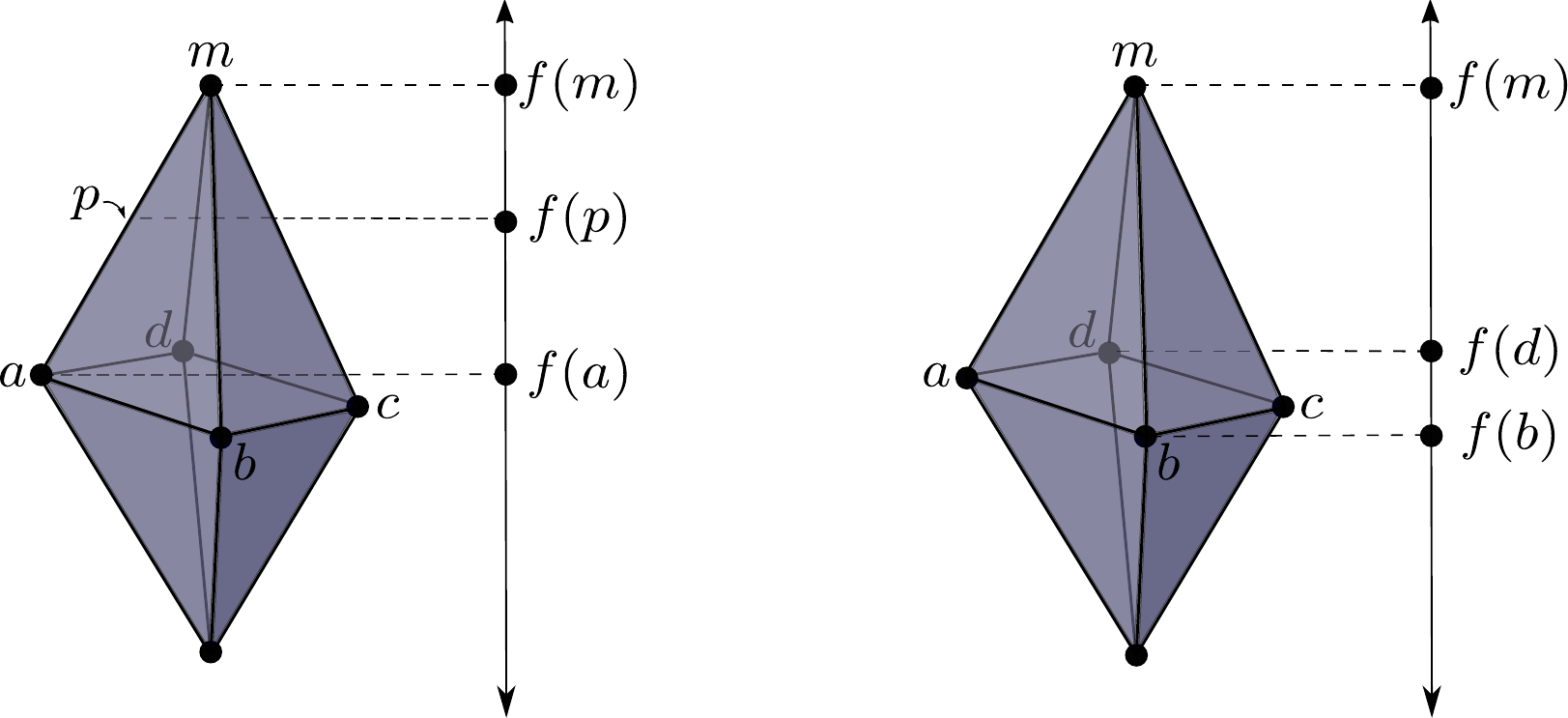}
\begin{caption}{The leftmost images illustrates why $p \in X$, an interior point, is regular and the rightmost image illustrates why $m \in X$ is a critical simplex, as described in \exref{icos}\label{fig:icos}}. 
\end{caption}
\end{figure}

\begin{prop}\label{prop:jacobi}
Let $f\colon X^n \to \RR^k$ be a generic PL function with $n \ge k$. The Jacobi set $J_f$ is an embedded simplicial complex in $X$ of dimension at most $k-1$.
\end{prop}

\begin{proof}
As the function $f$ is affine restricted to any given simplex, $\sigma$, if $p \in \mathring{\sigma}$ is critical than for any $q \in \mathring{\sigma}$, $q$ is also critical, i.e., criticality is constant on simplices. Also, by genericity of $f$, if $\sigma$ is a critical simplex, then $\dim \sigma \le k-1$.

To complete the proof, it is sufficient (and necessary) to show that criticality is a closed condition or equivalently that regularity is an open condition. Let $\tau$ be a face of $\sigma$ of one dimension lower than $\sigma$, $x \in \mathring{\tau}$, and $p \in \mathring{\sigma}$. We claim that if $x$ is regular, then $p$ is regular. Indeed, we have
\[
\Link (x) \cong \partial \tau \ast (p \amalg p') \ast S^{n-k-2} \quad \text{ and } \quad \Link (p) \cong \partial \sigma \ast S^{n-k-2},
\]
where $p'$ is an interior point of the other simplex containing $\tau$ as a codimension one face.
There is a PL homeomorphism $h\colon \Link(x) \cong \Link(p)$ that is the identity on the last factor of the join and such that $h(x) = p$.
Now if $x$ is regular via the decomposition $\Link(x) \cong S^{n-k-1} \ast S^{k-1}$, then the resulting decomposition of $\Link (p)$ proves that $p$ is regular. By induction on codimension, if the simplex $\sigma$ is critical, then any point in its closure is critical as well.
\end{proof}

\begin{example}\label{ex:delta3}
Consider a generic projection $\pi \colon \partial \Delta^3 \to \RR^2$ as in the leftmost image of \figref{delta3_regularity}.  As $\dim \partial \Delta^3 = \dim \RR^2$, this example illustrates the exceptional case in the definition of regularity.  Let $p$ be an interior point of the one-simplex $[ab]$ as indicated. The link of $p$ is PL circle given as the join $\{c,d\} \ast \{a,b\}$ which is the union of the one simplices $[ad] \cup [db] \cup [bc] \cup [ca]$. The restriction of $\pi$ to this circle is not a homeomorphism onto its image---it's not even injective--- so $p$, and hence the one simplex $[ab]$ are critical. 

Next, consider the point $r$, which is in the interior of the one-simplex $[ac]$.  This point is indeed regular, as the restriction of $\pi$ to its link---which is the circle $[ab]\cup [bc] \cup [cd] \cup [da]$---is a homeomorphism onto its image.

\end{example}

\begin{figure}[h!] 
\centering
\includegraphics[scale=.5]{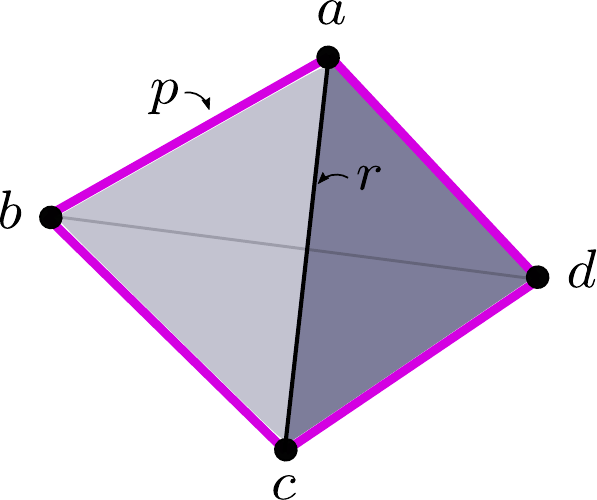}
\begin{caption}{A generic projection $\pi : \partial \Delta^3 \to \RR^2$ along with a regular point, $r$, and a critical point, $p$. The Jacobi set of this projection forms a subcomplex and is highlighted.\label{fig:delta3_regularity}}
\end{caption}
\end{figure}


\subsection{Three notions of regular point}

Let $f\colon X^n \to \RR^k$ be a generic PL function and $n \ge k$. There are (at least) three definitions for when a simplex $\sigma \subseteq X$ of dimension $k-1$ is critical. (From before, this dimension is the critical dimension.) To begin, let us recall the definition from above; in this subsection only, we will call this notion of critical {\em L-critical}. (Also, recall that being regular is an open condition and constant on simplices.)

\begin{definition}[$L$-Criticality] A $k-1$ dimensional simplex $\sigma \subseteq X$ is \emph{L-critical} if it is not $L$-regular, where
\begin{itemize}
\item $k=1$:  The point $p \in \sigma$ is $L$-regular if there exists a PL homeomorphism $\Link (p) \cong S^{n-k-1} \ast S^{k-1}$ such that
\[
f |_{\Link (p)} : S^{n-k-1} \ast S^{k-1} \to f(p) \ast S^{k-1}\]
respects the join decomposition and is a PL homeomorphism of spheres on the second factor; or
\item $k>1$: The simplex $\sigma$ is $L$-regular if there exists a point in the interior of $\sigma$ that is regular.
\end{itemize}
\end{definition}

Under our current hypotheses, a map $f$ has a well-defined derivative as explained in \S3.10 of \cite{thurston}. Using this PL differential one could make the following definition.

\begin{definition}[$D$-Criticality] A $k-1$ dimensional simplex $\sigma \subseteq X$ is \emph{D-critical} if
\begin{itemize}
\item $k=1$:  The differential at the point $\sigma$, $D_{\sigma} f$ is not surjective; or
\item $k>1$: There is a point $p \in \sigma$ for which $D_p f$ is not surjective.
\end{itemize}
\end{definition}

In \cite{edelsbrunner2008reeb}, the authors give a homological definition of criticality. To begin, let $\vec{u} \in S^{k-1}$ be a unit vector and define a function $h_{\vec{u}} \colon X \to \RR$ as $h_{\vec{u}} (p) = \langle f(p) , \vec{u} \rangle$. Let $\sigma \subseteq X$ be a $(k-1)$-simplex and $\vec{u}$ a unit normal to $f(\sigma)$, then $h_{\vec{u}}$ determines {\em upper} and {\em lower} subcomplexes of the link $\Link \sigma$. 

\begin{definition}[$H$-Criticality] The simplex $\sigma$ is \emph{H-critical} if its upper link has a non-vanishing reduced Betti number (with respect to $\ZZ/2$ coefficients). 
\end{definition}

The H-critical condition is actually symmetric in upper versus lower links as one can see by considering the unit vector $-\vec{u}$. 

All three of these notions of criticality are distinct. The H-critical condition has the advantage that it can algorithmically be implemented and checked by machine. In the absence of PL homology spheres and acyclic complexes that are not collapsible, H-criticality would agree with L-criticality; of course, such spaces do exist even in low dimension, e.g., the Poincar\'{e} sphere in dimension three or Bing's house in dimension~two.
 
 
 \begin{prop}
 Let $f: X \to \RR^k$ satisfy \ehp and $\sigma \subseteq X$ a $(k-1)$-simplex.
 \begin{itemize}
 \item[(a)] If $\sigma$ is D-critical, then $\sigma$ is L-critical;
 \item[(b)] If $\sigma$ is D-critical, then $\sigma$ is H-critical.
 \end{itemize}
 \end{prop}
 
 \begin{proof}
 We only prove (a) as (b) follows similarly. That D-criticality implies L-criticality is equivalent to L-regularity implying D-regularity. Let $p \in \mathring{\sigma}$ and assume that $\sigma$ is L-regular with corresponding join presentation $\Link (p) \cong S^{n-k-1} \ast S^{k-1}$.  Since $p$ is a regular point,  $f$ maps the link at $p$ onto the $k$-ball centered at $f(p)$ in a way that preserves the ``unit sphere" in the tangent space (the second factor in the join presentations), hence $Df$ is surjective at $p$.
\end{proof}
 
\begin{example}
The following illustrates a complex that has no D-critical vertices, but has H-critical vertices. Triangulate $S^1$ by four points. Consider the two sphere, $S^2 = pt \ast S^1 \ast pt$. Now define a function on $S^2$ that sends the two cone points to $0 \in \RR$ and assigns  $\pm 1$ to the vertices on the (equatorial) $S^1$ in an alternating way; extend this function linearly to obtain a PL map $f\colon S^2 \to \RR$. Neither cone point is $D$-critical, since the differential is surjective. However, the upper link of a cone point is homeomorphic to $S^0$, so the cone points are $H$-critical.
\end{example}

\begin{example}
Our final example illustrates that the PL differential may miss some L-critical points, even in the case where our combinatorial manifold is in fact smooth. Let $h\colon T^2 \to \RR$ be the upright height function on the torus. Consider a triangulation $\lvert K \rvert \cong T^2$ such that a neighborhood of the lower index one critical point is as in \figref{funkytorusmonkey}. The differential of the height function is clearly surjective at the saddle point. Yet the link can only be decomposed into an ``upper" $S^0$ and a ``lower" $S^0$, hence this point is L-critical.
\end{example}

\begin{figure}[h!] 
\centering
\includegraphics[scale=.75]{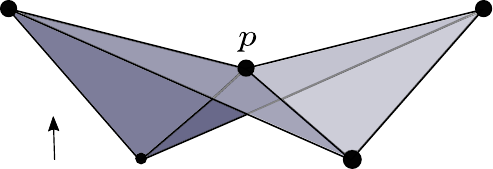}
\begin{caption}{This figure is a piece of a larger triangulated torus. Although the PL differential of the height function at the saddle point $p$ is surjective, the link of $p$ does not satisfy the condition of L-regularity.} \label{fig:funkytorusmonkey}
\end{caption}
\end{figure}

\section{An application to Reeb spaces}\label{sec:reeb}

Reeb graphs have long been an object of interest, and, roughly, are graphic summaries of level sets of real valued functions. The Reeb graph is formed by identifying points that lie in the same connected component of a level set.  First introduced in ~\cite{reeb1946points} for use in Morse theory, they have since been widely used in applications (see \cite{biasotti2008reeb} for a survey) and have been studied for their mathematical properties (e.g., \cite{bauer2014measuring, de2016categorified, cole2004loops}) as well as in terms of efficient computation (e.g., \cite{doraiswamy2009efficient, dey2013efficient}) Reeb spaces generalize the notion of Reeb graphs to the case of an arbitrary map between spaces. 

\begin{definition}\label{def:reeb} Let $f\colon X \to Y$ be a continuous map of spaces.  The \emph{Reeb space} $\WW_f$ is the quotient space defined by the relation: for $a,b \in X$, we have $a \sim b$ whenever $f(a) = f(b)$ and if $a$ and $b$ are in the same connected component of $f^{-1}(f(a))$. 
\end{definition}

Let $q: X \to \WW_f$ be the canonical quotient map. It is often useful to visualize the relationship between these spaces in diagramatic form:
\begin{center}
\begin{tikzcd}[column sep=small]
X \arrow[rr, "f"] \arrow[dr, "q"'] & & Y \\
& \WW_f \arrow[ur, "g"'] & 
\end{tikzcd}
\end{center}

\noindent
where $g$ is the {\em Stein factorization} of $f$, the unique continuous map making the diagram commute. 

\begin{lemma}
Let $\WW_f$ be the Reeb space corresponding to the map $f\colon X \to Y$. If $X$ is normal, then $\WW_f$ is Hausdorff.
\end{lemma}

\begin{proof}
Let $\sim = \{(a,b) \in X \times X \mid a \sim b\}$ and let $\sim^C $ denote the complement of $\sim$.
Since $X$ is normal, it suffices to show that $\sim^C$ is an open subset of $X \times X$. Let $(a,b) \in \sim^C$, meaning $a$ and $b$ do not lie in the same connected component of a level set of $f$. Note that the level set components containing $a$ and $b$, which we denote $A$ and $B$, respectively, are disjoint, and as they are connected components, both are closed in $X$. Since $A$ and $B$ are disjoint closed subsets of a normal space, we can find disjoint open sets containing them, which we denote $A'$ and $B'$, respectively. Then $A' \times B' \subseteq X \times X$ is an open set containing $(a,b)$. We claim that $A' \times B' \cap \sim = \emptyset$. Suppose, towards a contradiction, that $(x,y) \in A' \times B' \cap \sim$. Then $x$ and $y$ would be in the same connected component of a level set of $f$ while simultaneously being in $A'$ and $B'$, a contradiction. Since we have found an open neighborhood around an arbitrary point of $\sim^C$, we know $\sim$ is closed in $X \times X$, and thus, $\WW_f= X/\sim$ is Hausdorff.
\end{proof}

%
%
%
%
%
%

In our setting of PL functions $f \colon X \to \RR^k$, with $X$ a combinatorial manifold, by the above, the associated Reeb space $\WW_f$ will  Hausdorff.  Under the further hypotheses that $\dim X \ge k$ and that $f$ is generic, we know that the Jacobi set $J_f \subseteq X$ is an embedded simplicial complex by Proposition \ref{prop:jacobi}. The Jacobi set is stratified by its face poset, $\Nat(J_f)$, and Lemma \ref{lem:faceext} then determines two stratifications of $X$: the right conical extension which is stratified by $\Nat(J_f)^{\triangleright}$ and a refinement thereof given by $\Nat (J_f)^{\wwedge}$.

\subsection{The stratification of $\RR^k$}

Next, we will show that $f(J_f)$ can determine similar stratifications of $\RR^k$.

Since $J_f$ contains only $(k-1)$ and lower dimensional simplices and since $f$ is generic, we see that $f$ takes simplices to simplices, hence $f(J_f)$ is a collection of $(k-1)$-simplices in $\RR^k$. Even so, there is no guarantee that these $(k-1)$-simplices will not intersect in $\RR^k$, which precludes us from immediately endowing $f(J_f)$ with a simplicial structure. The first step to understanding why we \emph{can} do this is establishing the following lemma.

\begin{lemma}\label{lem:intersections}
Let $f\colon K \to \RR^k$ be a generic PL map from a finite simplicial complex comprised entirely of $(k-1)$-simplices. Then $P = \{p \in \RR^k \mid |f^{-1}(p)| > 1\}$ is a disjoint set of $(k-2)$-dimensional convex regions, and for every $p \in P$, we have $|f^{-1}(p)| = 2$. 
\end{lemma}

\begin{proof}
Since $f$ is generic and PL, the image of any $(k-1)$-simplex of $K$ is again an $(k-1)$ simplex. Since $K$ only consists of $(k-1)$-simplices, $P$ is exactly the intersections of these $(k-1)$-simplices in $\RR^k$ (not including the shared faces of simplices that were also adjacent in $K$). 
Any non-transverse intersection would be removable by an arbitrary perturbation of some vertex of $f(K)$, violating the genericity of $f$. Similarly, no triple of simplex images can intersect, since an arbitrary perturbation of a vertex of one of the simplices involved could change the intersection to three two-way intersections. Thus, intersections of images of $(k-1)$-simplices are pairwise transverse intersections, and therefore are convex with dimension $(k-2)$. Observe that since intersections are only pairwise, we have $|f^{-1}(p)|=2$.
If the collection of such intersections were not disjoint, this would imply an intersection three simplex images, which is not possible by the reasoning above.
\end{proof}

Then, as a corollary of \lemref{intersections}, we have the following.

\begin{prop}\label{prop:Jsimp}
If $f\colon X^n \to \RR^k$ is a generic PL map, and $n \ge k$, then $f(J_f)$ admits a refinement into simplices so that, for $\tau, \sigma \in f(J_f)$ where $\dim(\tau) \leq \dim(\sigma)$, whenever a point $x \in \tau$ and $x \in \sigma$, we have $\tau \subseteq \sigma$ as subsets of $\RR^k$.
\end{prop}

 Since points of $f(J_f)$ with more than one preimage form $(k-2)$-dimensional convex polytopes, we can triangulate these intersection regions so that $f(J_f)$ is a collection of simplices. Although there are many ways to triangulate a convex polytope (and the method of triangulation does not effect our results), a simple explicit triangulation is to place a Steiner point at the barycenter of each $i$-dimensional face to star-triangulate each $i$-dimensional face (whenever the face is not already an $i$-simplex), proceeding from $i=2,3,\ldots,k-2$. 

An example of $f(J_f)$ along with its refinement into simplices is given in \figref{noncomplex}. 

Note that although \propref{Jsimp} does not generally imply $f(J_f)$ is a simplicial complex, it does maintain the property that every simplex of this refinement is either disjoint from, or a subset/superset of every other simplex in the refinement, and thus is a natural candidate for stratification by containment (see \defref{cP}). 

\begin{figure}[h!]
\centering
\includegraphics[scale=.75]{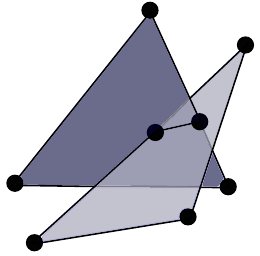}
\begin{caption}{The above is an example of an intersections in $f(J_f)$ for some map $f\colon X \to \RR^3$. By adding vertices and and edges on intersection, we refine $f(J_f)$ into eight vertices, seven edges, and two faces. This can then be stratified by containment, as in \propref{Jsimp}. Note that, although more refinement may be necessary for maps into higher dimensional Euclidean space, no further refinement is necessary for $k \leq 3$ since intersections are one- or zero-dimensional, and are thus already guarenteed to be simplices.\label{fig:noncomplex}}
\end{caption}
\end{figure}

 In what follows, we may abuse notation and use $f(J_f)$ to mean $f(J_f)$ taken with this refinement.
Then, we can define the following.

\begin{definition} \label{def:cP}
Let $n \ge k$, and let $f\colon X^n \to \RR^k$ be a generic PL map.  Further,  suppose that~$f(J_f)$ has been refined into simplices as described in \propref{Jsimp}. We define the poset $\cP_f$, where, for $\tau, \sigma \in f(J_f)$, we have $\tau \leq \sigma$ if and only if $\tau \subseteq \sigma$ as subsets of $\RR^k$. We stratify $f(J_f)$ via the map $\varphi_{f(J_f)}$, where, if $\tau$ is the lowest dimensional simplex containing $x \in f(J_f)$, then $\varphi_{f(J_f)} (x) = \tau$. 
\end{definition}

Given the stratification $f(J_f) \xrightarrow{\varphi_{f(J_f)}} \cP_f$ as above, we can stratify $\RR^k$ via our connected ambient stratification
$\RR^k \xrightarrow{\psi_{f(J_F)}} \cP_f^{\wwedge}$ (see \defref{connectedstrat}) so that each stratum is a connected manifold.

\subsection{The stratification of $\WW_f$}\label{ssec:reebstrat}
Since $\WW_f$ is a quotient of $X$ we would like to use the universal property to stratify $\WW_f$. However, the stratifying map $\psi_{J_f} \colon X \to \Nat (J_f)^{\wwedge}$ is not constant on the fibers of the quotient map $q$. (Nor does the map $f$ ---in general---satisfy the conditions under which $\Nat^{\wwedge}$ is functorial.)

Consider the composition 
\[
 \varphi_{f(J_f)} \circ f  |_{J_f} \colon J_f \to \cP_f .
\]
This stratification is a refinement of $\Nat(J_f)$. For clarity, let us denote a copy of $\cP_f$ by $\widetilde{\cP}_f$ and notate the preceding map by $\Psi_{J_f}$, so we have $\Psi_{J_f} \colon J_f \to \widetilde{\cP}_f$. We can now apply the connected ambient construction to stratify $\WW_f$:
\[
\Psi_{\WW_f} \colon \WW_f \to \widetilde{\cP}_f^{\wwedge}.
\]
As the Stein factorization $g\colon \WW_f \to \RR^k$ is continuous---hence preserving connected components---it induces a natural map of posets $\overline{g} \colon \widetilde{\cP}_f^{\wwedge} \to \cP_f^{\wwedge}$. Moreover, the following is an easy tour through definitions.


\begin{prop}\label{prop:steinstrat}
The Stein factorization $g\colon \WW_f \to \RR^k$ is a map of stratified spaces, i.e., the following commutes in the category of topological spaces
\[
\xymatrix{ \WW_f \ar[r]^g \ar[d]_{\Psi_{\WW_f}} & \RR^k \ar[d]^{\psi_{\RR^k}} \\ \widetilde{\cP}_f^{\wwedge} \ar[r]^{\overline{g}} & \cP_f^{\wwedge}}
\]
\end{prop}

\begin{example}
Consider the height function on the upright torus (equipped with a triangulation), $h \colon T^2 \to \RR$ with  Jacobi set, $J_h$, the four Morse critical points.  \figref{torusheight} depicts the associated Stein factorization and stratifying posets. In this case, one could also conclude that the Stein factorization is a stratified map by applying Proposition \ref{prop:facemap}. 
\end{example}

\begin{figure}[h!] 
\centering
\includegraphics[scale=.40]{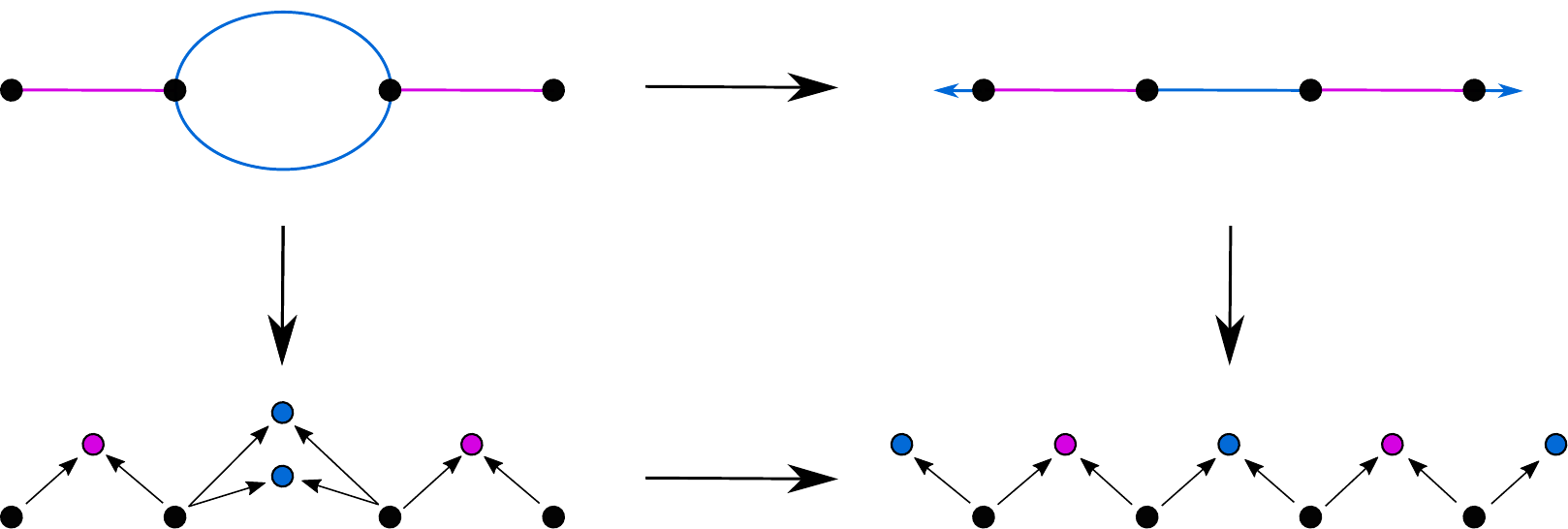}
\begin{caption}{The Stein factorization as a stratified map for the height function on the upright torus. This is an example of the diagram in \propref{steinstrat}.}
\end{caption}
\label{fig:torusheight}
\end{figure}

\begin{prop}
Let $n \ge k$. If the generic PL map $f \colon X^n \to \RR^k$ is injective on the $(k-1)$-skeleton of $X$, then $f \colon X \to \RR^k$ and $q \colon X \to \WW_f$ are stratified maps.
\end{prop}

\begin{proof}
If $f$ is injective on the $(k-1)$-skeleton of $X$, then none of the images of $(k-1)$-simplices of $J_f$ will intersect transversely (i.e., $f(J_f)$ has a simplicial complex structure), so $\cP_f^{\wwedge} = \Nat(f(J_f))^{\wwedge}$. Furthermore, we have $\Nat(f(J_f)) = \Nat(J_f)$ meaning that restricting $f$ and $q$ to elements of $\Nat(J_f)^{\wwedge}$ defines surjections onto $\Nat(f(J_f))^{\wwedge}$. 
\end{proof}

\begin{example}
Let us return to Example \ref{ex:delta3}, so we consider a generic projection $\pi \colon \partial \Delta^3 \to \RR^2$.
The Jacobi set is the simplicial circle formed by  $[ab]\cup [bc] \cup [cd] \cup [da]$ and is indicated by pink bolded lines. The Reeb space is shown on the right, and is homeomorphic to the 2-disc, $\WW_f \cong D^2$: it is stratified with four zero-simplices, four one-simplices (the images of those in the Jacobi set), and the open disc as its top stratum.
\end{example}

\begin{figure}[h!] 
\centering
\includegraphics[scale=.5]{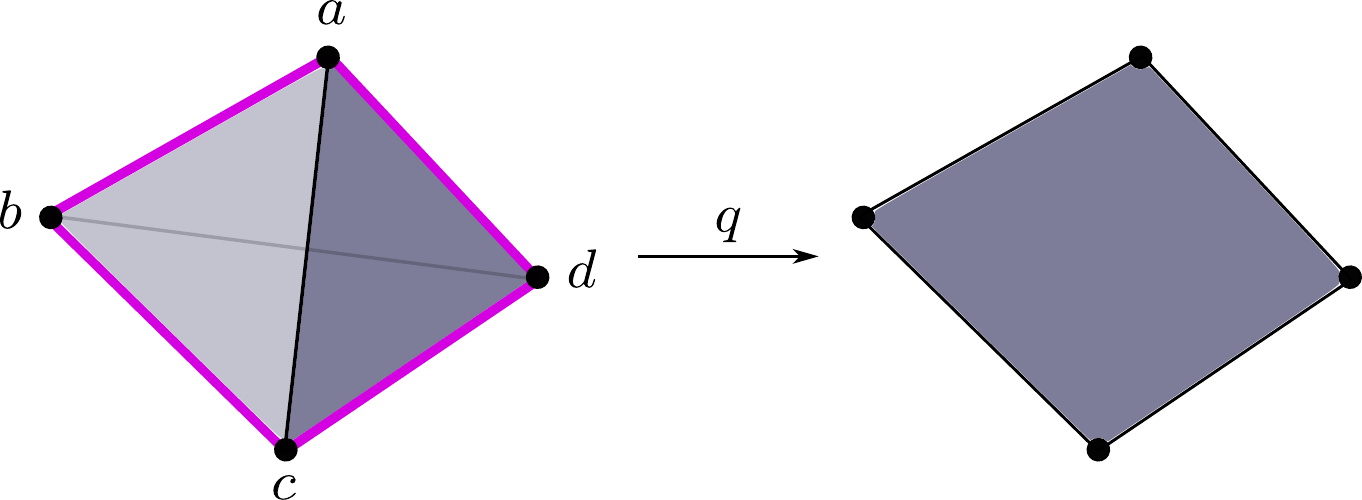}
\begin{caption}{The quotient map and stratification of the Reeb space corresponding to a generic projection $\pi \colon \partial \Delta^3 \to \RR^2$. The Jacobi set, $J_\pi$, is highlighted as a subcomplex of the domain.\label{fig:delta3}}
\end{caption}
\end{figure}

\subsection{Comparing to the EHP stratification}

A main result of \cite{edelsbrunner2008reeb} is defining the so called ``canonical stratification'' of Reeb spaces, which we will refer hereafter as the \emph{EHP stratification}. The EHP stratification is described using an algorithmic construction. First, $X$ is decomposed by coarsening a decomposition of $X$ into prisms created by preimages of the affine hulls of images of $(k-1)$-simplices. This decomposition is inherited by $\WW_f$, and the EHP stratification of $\WW_f$ is the skeletal stratification of a triangulation of the decomposition. As noted in \cite{edelsbrunner2008reeb} a precise version of this algorithm is limited to $k < 5$, since a Boolean subroutine used to determine when two triangulated spaces of dimension $k-1$ are homeomorphic is undecidable for $k \geq 5$. 

Although not explicitly considered in \cite{edelsbrunner2008reeb}, the process by which $\WW_f$ is triangulated could be further used to triangulate both $X$ and $\RR^k$ by star triangulating various convex subspaces. One of the more obvious differences between the EHP stratification of $\WW_f$ and our stratification in \ssref{reebstrat} is that, by construction, each stratum of $\widetilde{\cP}_f^{\wwedge}$ is connected, whereas this is not true in general for a stratum of the EHP stratification. This difference is relatively minor, and we expect under further hypotheses (e.g., $J_f$ is defined equivalently when using homological-criticality vs. criticality and $f(J_f)$ is already a simplicial complex without further refinement) that the stratification of \cite{edelsbrunner2008reeb} and the corresponding skeletal filtration obtained via Proposition \ref{prop:filtration} agree.

\bibliographystyle{plain}
\bibliography{references}

\begin{thebibliography}{10}

\bibitem{aft}
David Ayala, John Francis, and Hiro~Lee Tanaka.
\newblock Local structures on stratified spaces.
\newblock {\em Advances in Mathematics}, 307:903--1028, 2017.

\bibitem{bauer2014measuring}
Ulrich Bauer, Xiaoyin Ge, and Yusu Wang.
\newblock Measuring distance between reeb graphs.
\newblock In {\em Proceedings of the thirtieth annual symposium on
  Computational geometry}, pages 464--473, 2014.

\bibitem{biasotti2008reeb}
Silvia Biasotti, Daniela Giorgi, Michela Spagnuolo, and Bianca Falcidieno.
\newblock Reeb graphs for shape analysis and applications.
\newblock {\em Theoretical computer science}, 392(1-3):5--22, 2008.

\bibitem{bjorner1984posets}
Anders Bj{\"o}rner.
\newblock Posets, regular cw complexes and bruhat order.
\newblock {\em European Journal of Combinatorics}, 5(1):7--16, 1984.

\bibitem{cole2004loops}
Kree Cole-McLaughlin, Herbert Edelsbrunner, John Harer, Vijay Natarajan, and
  Valerio Pascucci.
\newblock Loops in reeb graphs of 2-manifolds.
\newblock {\em Discrete \& Computational Geometry}, 32(2):231--244, 2004.

\bibitem{de2016categorified}
Vin De~Silva, Elizabeth Munch, and Amit Patel.
\newblock Categorified reeb graphs.
\newblock {\em Discrete \& Computational Geometry}, 55(4):854--906, 2016.

\bibitem{dey2013efficient}
Tamal~K Dey, Fengtao Fan, and Yusu Wang.
\newblock An efficient computation of handle and tunnel loops via reeb graphs.
\newblock {\em ACM Transactions on Graphics (TOG)}, 32(4):1--10, 2013.

\bibitem{doraiswamy2009efficient}
Harish Doraiswamy and Vijay Natarajan.
\newblock Efficient algorithms for computing reeb graphs.
\newblock {\em Computational Geometry}, 42(6-7):606--616, 2009.

\bibitem{edelsbrunner2008reeb}
Herbert Edelsbrunner, John Harer, and Amit~K Patel.
\newblock Reeb spaces of piecewise linear mappings.
\newblock In {\em Proceedings of the twenty-fourth annual symposium on
  Computational geometry}, pages 242--250, 2008.

\bibitem{Friedman}
Greg Friedman.
\newblock {\em Singular intersection homology}, volume~33 of {\em New
  Mathematical Monographs}.
\newblock Cambridge University Press, Cambridge, 2020.

\bibitem{hersh2021shellability}
Patricia Hersh and Richard Kenyon.
\newblock Shellability of face posets of electrical networks and the cw poset
  property.
\newblock {\em Advances in Applied Mathematics}, 127:102178, 2021.

\bibitem{lurie}
Jacob Lurie.
\newblock Higher algebra.
\newblock available at
  \href{https://www.math.ias.edu/~lurie/papers/HA.pdf}{Author's Homepage},
  2022.

\bibitem{liviu}
Liviu Nicolaescu.
\newblock {\em An invitation to {M}orse theory}.
\newblock Universitext. Springer, New York, second edition, 2011.

\bibitem{Nocera}
Guglielmo Nocera and Marco Volpe.
\newblock Whitney stratifications are conically smooth.
\newblock {\em arXiv preprint arXiv:2105.09243}, 2021.

\bibitem{reeb1946points}
Georges Reeb.
\newblock Sur les points singuliers d'une forme de pfaff completement
  integrable ou d'une fonction numerique [on the singular points of a
  completely integrable pfaff form or of a numerical function].
\newblock {\em Comptes Rendus Acad. Sciences Paris}, 222:847--849, 1946.

\bibitem{RSPL}
Colin~Patrick Rourke and Brian~Joseph Sanderson.
\newblock {\em Introduction to piecewise-linear topology}.
\newblock Springer-Verlag, New York-Heidelberg, 1972.
\newblock Ergebnisse der Mathematik und ihrer Grenzgebiete, Band 69.

\bibitem{stanley1991f}
Richard~P Stanley.
\newblock f-vectors and h-vectors of simplicial posets.
\newblock {\em Journal of Pure and Applied Algebra}, 71(2-3):319--331, 1991.

\bibitem{thurston}
William~P. Thurston.
\newblock {\em Three-dimensional geometry and topology. {V}ol. 1}, volume~35 of
  {\em Princeton Mathematical Series}.
\newblock Princeton University Press, Princeton, NJ, 1997.
\newblock Edited by Silvio Levy.

\bibitem{Trotman}
David Trotman.
\newblock Stratification theory.
\newblock In {\em Handbook of geometry and topology of singularities. {I}},
  pages 243--273. Springer, Cham, [2020] \copyright 2020.

\end{thebibliography}

\end{document}